\documentclass{amsart}
\usepackage{amsmath,amsfonts, mathtools,amssymb,latexsym,amssymb,amsthm,adjustbox,mathrsfs,amsbsy, bm,tikz-cd,indentfirst, makecell,graphicx, verbatim, calc, enumerate,xy, hyperref}
\usepackage[top=1in,bottom=1.3in,left=1.1in,right=1.1in]{geometry}
\usepackage{fancyhdr}
\hypersetup{colorlinks=true,
    linkcolor=blue,
    filecolor=magenta,      
    urlcolor=cyan,}
               
\theoremstyle{plain}
\newtheorem{theorem}{Theorem}[section]
\newtheorem{lemma}[theorem]{Lemma}

\newtheorem{corollary}[theorem]{Corollary}
\newtheorem{proposition}[theorem]{Proposition}
\newtheorem{remark}[theorem]{Remark}
\newtheorem{definition}[theorem]{Definition}

\newtheorem{point}[theorem]{}

\newcommand{\m}{\mathfrak{m}}

\newcommand{\n}{\mathfrak{n}}

\newcommand{\injdim}{\operatorname{injdim}}

\newcommand{\Ass}{\operatorname{Ass}}
\newcommand{\height}{\operatorname{height}}
\newcommand{\maxsp}{\operatorname{MaxSpec}}
\newcommand{\Supp}{\operatorname{Supp}}
\newcommand{\variety}{\operatorname{V}}

\newcommand{\Spec}{\operatorname{Spec}}

\newcommand{\Mod}{\operatorname{Mod}}

\newcommand{\Ext}{\operatorname{Ext}}

\usepackage{graphicx}
\usepackage{subfigure}
\usepackage{epsfig}
\usepackage{epstopdf}
\usepackage{float} 
\theoremstyle{plain}

\usepackage{xcolor}
\usepackage{titlesec}

\titleformat{name=\section}{}{\thetitle.}{0.5em}{\centering\scshape\large}

\begin{document}
\title{\large \textbf{Bounds on Bass numbers of local cohomology modules }}
\author{\textsc{Sayed Sadiqul Islam}}
\address{Department of Mathematics, IIT Bombay, Powai, Mumbai 400076, India}
\email{ssislam1997@gmail.com, 22d0786@iitb.ac.in}
\author{\textsc{Tony J. Puthenpurakal}}
\address{Department of Mathematics, IIT Bombay, Powai, Mumbai 400076, India}
\email{tputhen@math.iitb.ac.in}
\date{\today}

    
	\begin{abstract}
    Let $R=K[x_1,\ldots,x_m]$ where $K$ is an uncountable algebraically closed field of characteristic $0$.
	  For a prime ideal $P$ of $R$, let $
    \mu_j(P,M)$ be the $j$-th Bass number of an $R$-module $M$ with respect to the prime $P$. For $1\leq g\leq m-1$, we construct a set $\mathcal{S}_g(t)$ such that  $\mathcal{S}_g(t)\subseteq \mathcal{S}_g(t+1)$ for all $t\geq 1$ and $\bigcup_{t\geq 1} \mathcal{S}_g(t)=\Spec_g(R)=\{P\in \Spec (R)\mid \height P=g\}$. Let $\mathcal{T}$ be a Lyubeznik functor on $\Mod(R)$. We prove that there exists some function $\phi^g_i: \mathbb{N}^2\rightarrow \mathbb{N}$ which is monotonic in both the variables such that  $\mu_i(P,\mathcal{T}(R))\leq \phi^g_i(e(\mathcal{T}(R)),t)$ for all $P\in \mathcal{S}_{g}(t)$. In particular, the result holds for composition of local cohomology functors of the form $ H^{i_1}_{I_1}(H^{i_2}_{I_2}(\dots H^{i_r}_{I_r}(-)\dots)$.
	\end{abstract}
    
	 \maketitle
    
	\section{Introduction}
	Throughout this paper, $R$ is a commutative Noetherian ring. For a locally closed subscheme of $\Spec(R)$, let $H^i_Y(M)$ be the $i$-th local cohomology module of $M$ supported in $Y$. $H^i_Y(M)$ is denoted by $H^i_I(M)$ when $Y$ is closed in $\Spec(R)$, defined by an ideal $I$ of $R$.  
    
    For a prime ideal $P$ of $R$, the $j$-th Bass number of an $R$-module with respect to the prime $P$ is defined as $\mu_j(P,M)=\dim_{\kappa(P)}\Ext^j_{R_P}(\kappa(P),M_P)$. Let the injective dimension of $M$ be denoted by $\injdim_R M$. The local cohomology modules have been studied by a number of authors. Although these modules have been studied extensively, their structure remains poorly understood. Most of the time, we do not know when they vanish. Furthermore, when they do not vanish, they are rarely finitely generated even if $M$ is. For example, if $(R,\mathfrak{m})$ is a local ring of dimension $d\geq 1$, then $H^d_\mathfrak{m}(R)$ is never finitely generated. Therefore, an important problem is to identify some finiteness properties of local cohomology modules for better understanding of their structure. In \cite{HS-93}, Huneke and Sharp, for regular rings containing a field of characteristic $p>0$ proved the following results:
    \begin{enumerate}
        \item $H^j_\m(H^i_I(R))$ is injective, where $\m$ is any maximal ideal of $R$.
        \item $\injdim H^i_I(R)\leq \dim \Supp H^i_I(R)$.
        \item  The set of associated primes of $H^i_I(R)$ is finite.
        \item  All the Bass numbers of $H^i_I(R)$ are finite.
    \end{enumerate}
	Following these results, Lyubeznik in his seminal paper \cite{Lyu-93} proved analogous finiteness properties of local cohomology modules in characteristic $0$. He used the theory of $\mathcal{D}$-modules to prove the results for considerably large class of functors, known as Lyubeznik functor (See, \ref{lyubeznik functor} for the definition). The local cohomology and the composition of local cohomologies are examples of Lybeznik functors. The natural question is whether the results of Huneke and Sharp can be extended to the Lyubeznik functor. In \cite{Lyu-97}, Lyubeznik developed the theory of $F$-modules over regular rings of characteristic $p>0$ and  proved similar results. He also used $\mathcal{D}$-modules and some new ideas to prove these results for unramified regular local ring (See, \cite{Lyu-2000}). An ingenious blend of $\mathcal{D}$-module and $F$-module theory by Bhatt et al. in \cite{BBLSZ-14} was used to show that, for smooth $\mathbb{Z}$-algebras, local cohomology modules have only finitely many associated primes.

In this paper, we are mostly concerned with  upper bounds of Bass numbers for local cohomology modules (more generally of $\mathcal{T}(R)$, where $\mathcal{T}$ is a Lyubeznik functor).  Let $e(M)$ denote the multiplicity of an $A_m(K)$-module $M$ where $A_m(K)$ is the Weyl algebra corresponding to the polynomial ring $K[x_1,\ldots,x_m]$. We now state the results proved in this paper. The first result we prove is the following.\medskip

\noindent
\textbf{Theorem A.} (See, Theorem \ref{Bounded bass no for m})  \phantomsection\label{main theorem for m}
 \textit{Let $R=K[x_1,\ldots,x_m]$ where $K$ is a field of characteristic $0$. Let $T(R)= H^{i_1}_{I_1}(H^{i_2}_{I_2}(\dots H^{i_r}_{I_r}(R)\dots)$. Suppose $\injdim T(R)=c$. Then for all $i=0,1,\ldots,c$; $$\mu_i(\m ,T(R))\leq \binom{m}{i}(1+i)^n e(T(R)) $$ for all  maximal ideal $\m$ of $R$.}
\medskip

This improves an earlier result from \cite{Put-14} (See, Proposition 1.3), but for the composition of local cohomology functors. In the theorem mentioned above, we are able to remove the hypothesis that $K$ is algebraically closed which was initially assumed in  \cite[Proposition 1.3]{Put-14}.

One naturally asks, what happens when instead of maximal ideals, we consider prime ideals. The paper provides some partial answers to this question. In \cite{BBY-15}, Banerjee et al. proved the characteristic $p$ version of the above theorem, where prime ideals are considered instead of maximal ideals. They proved that if $M$ is a $F$-finite, $F$-module over a Noetherian regular ring $S$ of finite Krull dimension, then there exists some $B> 0$ such that $\mu_i(P,M)\leq B$ for all $P\in \Spec(S)$, and $i\in \mathbb{N}$. 

In the next few results, we study the Bass numbers of $T(R)$ with respect to prime ideals.\medskip

\noindent
\textbf{Theorem B.} (See, Theorem \ref{proof of graded case m-1})  \phantomsection\label{main theorem for P but graded}
	\textit{Let $K$ be a field of characteristic 0 and $R$ be a standard graded polynomial ring $K[x_1,\ldots,x_m]$, i.e., $\deg(x_i)=1$ for all $i$.  Let $I_1,I_2,\ldots,I_r$  be homogeneous ideals of $R$ and $T(R)= H^{i_1}_{I_1}(H^{i_2}_{I_2}(\dots H^{i_r}_{I_r}(R)\dots)$. Then, there exists some $B> 0$ $\left(\text{depending on}\  T(R)\right)$ such that $$\mu_i(P,T(R))\leq B $$ for all  graded prime ideal $P$ of $ R $ of height $m-1$.}
	\medskip
    
    When we are not in a graded setup, the situation is different. We state few results that answer this question. For the remaining of the results we assume $K$ to be an uncountable algebraically closed field of characteristic $0.$\medskip

    \noindent
    \textbf{Lemma C.} (See, Lemma \ref{Result of mu_0})  \phantomsection\label{main theorem non graded}
       \textit{Let $K$ be an uncountable algebraically closed field of characteristic $0$ and $R=K[x_1,\ldots,x_m]$. Let $\mathcal{T}$ be a Lyubeznik functor on $\Mod(R)$. Then, $$\mu_0(P,\mathcal{T}(R))\leq 2^mm\ e(\mathcal{T}(R))$$ for all prime ideal $P$ of $R$ of height $m-1$.}
    \medskip

    Eisenbud and Evans proved that every ideal in $K[x_1,\ldots,x_m]$ can be generated by $m$ elements up to radical; see
\cite[Theorem 1]{EG-73}. Therefore, if $P$ is a prime ideal in $K[x_1,\ldots,x_m]$, then there exists some ideal $J$ such that  $ P=\sqrt{J}$ where $J=(f_1,\ldots,f_m)$. Hence, we naturally define the following set.
\begin{definition} For a fixed number $t$,
 $$H_{g}(t):=\{P\in \Spec(R)\ |\  \height(P)=g,\  P=\sqrt{J} \ \text{where\
 $J=(f_1,\ldots,f_m)$,\ $\deg(f_j)\leq t$ \}}.$$
 \end{definition}
It is easy to see that $H_{g}(t)\subseteq H_{g}(t+1)$ for all $t\geq 1$ and $\bigcup_{t\geq 1}H_{g}(t)=\Spec_{g}(R):=\{P\in \Spec(R)\ |\  \height P=g\}$. Since $\Spec_{g}(R)$ is uncountable, there exists some $t_0$ such that $H_{g}(t_0)$ is uncountable. Hence, $H_g(t)$ is uncountable for all $t\geq t_0$.
 
  As a consequence of \hyperref[main theorem non graded]{Lemma C}, we prove the following result which deals with the higher Bass numbers of $\mathcal{T}(R)$ for primes in $ H_{m-1}(t)$.\medskip

  \noindent
    \textbf{Theorem D.} (See, Theorem \ref{corollary for m-1})  \phantomsection\label{main theorem corollary}
  \textit{ Let $R=K[x_1,\ldots,x_m]$ where $K$ is an uncountable algebraically closed field of characteristic $0$.  Let $\mathcal{T}$ be a Lyubeznik functor on $\Mod(R)$. Then, for all $P\in H_{m-1}(t)$, $$\mu_i(P,\mathcal{T}(R))\leq \phi_i^{m-1}(e(\mathcal{T}(R)),t)$$ 
   where $$\phi_i^{m-1}(u,v)=2^mm\binom{m}{i}u(1+iv)^m.$$}\medskip

Next, we define $\mathcal{S}_g(t)$ for $1\leq g\leq m-1$ inductively.
\begin{definition}
Define $\mathcal{S}_{m-1}(t):=H_{m-1}(t)$. If $S_{g+1}(t)$ is defined for some $1\leq g\leq m-2$, then we define $\mathcal{S}_g(t)$ in the following way:
 $$\mathcal{S}_g(t):=\{P\in H_g(t)\mid\  C^P_g(t)\  \text{is uncountable}\}.$$ where 
 $$C^P_g(t):=\{Q\in S_{g+1}(t) \mid Q\supsetneq P\}$$
 \end{definition}
 The following properties of the sets $\mathcal{S}_g(t)$ for $1\leq g\leq m-2$ justify the way they are defined.
 \medskip

  \noindent
    \textbf{Proposition E.} (See, Proposition \ref{proposition of s_g proof})\phantomsection\label{Proposition about S_g}
  \textit{Let $1\leq g\leq m-2$. Then,
  \begin{enumerate}[\rm (1)]
      \item Let $P\in \mathcal{S}_g(t) $. Then, $C^P_g(t)\subseteq C^P_g(t+1)$ for all $t\geq 1$ and $\bigcup_{t\geq 1}C^P_g(t)=\Spec_{g+1}(R)\cap \variety(P)$;
      \item $\mathcal{S}_g(t)\subseteq \mathcal{S}_g(t+1)$ for all $t\geq 1$ and $\bigcup_{t\geq 1} \mathcal{S}_g(t)=\Spec_g(R)$.
  \end{enumerate}}\medskip
 
The final and main result result of this paper is the following.\medskip

\noindent
\textbf{Theorem F.} (See, Theorem \ref{proof of main theorem}) \label{bounded by some function}
 \textit{  Let $1\leq g\leq m-1$ and let $R=K[x_1,\ldots,x_m]$ where $K$ is an uncountable algebraically closed field of characteristic $0$. Let $M=\mathcal{T}(R)$ for some Lyubeznik functor $\mathcal{T}$. Then, there exists some function $\phi^g_i: \mathbb{N}^2\rightarrow \mathbb{N}$ which is monotonic in both the variables such that  $\mu_i(P,M)\leq \phi^g_i(e(M),t)$ for all $P\in \mathcal{S}_{g}(t)$.}\medskip

We provide a brief overview of the organization of the article. Section 2 introduces the necessary preliminaries - including definitions, assumptions, and supporting results needed to prove our theorems. In Section 3, we present the proof of \hyperref[main theorem for m]{Theorem A}, followed by the proof of \hyperref[main theorem for P but graded]{Theorem B} in Section 4. In Section 5, we prove  \hyperref[main theorem non graded]{Lemma C} and \hyperref[main theorem corollary]{Theorem D}. Finally, Section 6 completes the paper with the proof of \hyperref[bounded by some function]{Theorem F}.
	\section{Preliminaries}

	Let $R$ be a Noetherian commutative ring and $G$ be a finite subgroup of $\operatorname {Aut}(R)$; the group of automorphism of $R$. Further assume that $|G|$ is invertible in $R$. Let $R^G$ be the ring of invariants of $G$. By $R*G$,
    we denote the skew group ring. For basic properties of skew group ring we refer \cite{Put-16}. Let us recall few results from \cite{Put-16} that we need in this paper.
	\begin{corollary}\cite[Corollary 3.3]{Put-16}
		Let $I,I_1,\ldots,I_r$ be ideals in $R^G$. Then:
		\begin{enumerate}[\rm (1)]
		\item $H^i_{IR}(R)$ is an $R*G$-module  and $H^i_{IR}(R)^G\cong H^i_I(R^G)$ for each $i\geq 0$.
		\item For all $i_j\geq 0$, where $j=1,\ldots,r$, $H^{i_1}_{I_1}(H^{i_2}_{I_2}(\dots H^{i_r}_{I_r}(R)\dots)$ is an $R*G$-module and $$H^{i_1}_{I_1R}(H^{i_2}_{I_2R}(\dots H^{i_r}_{I_rR}(R)\dots)^G\cong H^{i_1}_{I_1}(H^{i_2}_{I_2}(\dots H^{i_r}_{I_r}(R^G)\dots).$$
		\end{enumerate}
	\end{corollary}
	The following result is crucial.
	\begin{theorem}\label{Bass number relation}\cite[Theorem 6.1]{Put-16}
		Let $K$ be a field and  $R$ be a regular domain containing $K$. Let $G$ be a finite subgroup of the group of automorphism of $R$. Let $|G|$ be invertible in $K$. Assume that $R^G$ is Gorenstein and $I_1,I_2,\ldots, I_r$ be ideals in $R^G$. Set $T(R^G)=H^{i_1}_{I_1}(H^{i_2}_{I_2}(\dots H^{i_r}_{I_r}(R^G)\dots)$ and $T(R)= H^{i_1}_{I_1R}(H^{i_2}_{I_2R}(\dots H^{i_r}_{I_rR}(R)\dots)$. Let $P$ be a prime ideal in $R^G$. Then $\mu_j(P,T(R^G))=\mu_j(P',T(R))$ where $P'$is any prime in $R$ lying over $P$.
	\end{theorem}
	Throughout the article $K$ is assumed to be  a field of characteristic $0$ unless explicitly stated otherwise. Let $R=K[x_1,\ldots,x_m]$ be the ring of polynomials in $m$ variables over $K$ and $A_m(K)$ be the corresponding  Weyl algebra. We mainly follow \cite[Chapter 1]{Bjo-79} for the theory of finitely generated left $A_m(K)$-modules. For finite generated left $A_m(K)$-modules $M$, let $l_{A_m(K)}(M)$ and $e_{A_m(K)}(M)$ denote the length and multiplicity respectively. When the ring $A_m(K)$ is clear from the context, we omit the subscript $A_m(K)$ and simply write $l(M)$ and $e(M)$. Let $\mathcal{B}_m$ be the class of holonomic $A_m(K)$-modules. In Bj\"{o}rk’s book \cite{Bjo-79}, holonomic modules are referred to as modules in the Bernstein class. 
	\begin{point}\normalfont
	     \textbf{Some properties of elements of $\mathcal{B}_m$:}
	Let $M\in \mathcal{B}_m$ and $R=K[x_1,\ldots,x_m]$. Then,
	\begin{enumerate}[\rm(1)]
		\item $M$ has finite length as left $A_m(K)$-modules. In fact, every strictly increasing sequence of $A_m(K)$-modules in $M$ contains at most $e(M)$ elements; see \cite[Chapter 1, Proposition 5.3]{Bjo-79}. 
        \item $\mathcal{B}_m$ is stable under extension. That is, if $0\rightarrow M_1\rightarrow M_2\rightarrow M_3\rightarrow 0$ is an exact sequence of left $A_m(K)$-modules, then $M_2\in \mathcal{B}_m$ if and only if $M_1$ and $M_3$ both $\in \mathcal{B}_m$; see \cite[Chapter 1, Proposition 5.2]{Bjo-79}. 
		\item Let $f\in R$. Then $R_f$ is a left $A_m(K)$-module and $R_f\in \mathcal{B}_m$; see \cite[Chapter 1, Theorem 5.5]{Bjo-79}. 
		\item The natural localization map $R\rightarrow R_f$ is a morphism of left $A_m(K)$-modules. Let $I\subseteq R$ be an ideal of $R$ generated by $\underline f= f_1,f_2,\ldots,f_r\in R$. Then the \v Cech complex of $R $ with respect to  $\underline f$ is defined by 
		$$
		\Check{C}^{\bullet}(\underline f,R):\  0\rightarrow R\rightarrow \bigoplus_iR_{f_i}\rightarrow \bigoplus_{i,j}R_{f_if_j}\rightarrow\ldots \rightarrow R_{f_1\ldots f_r}\rightarrow 0
		$$
		where the maps on every summand are localization map up to a sign. The local cohomology module of $R$ with support on $I$  is defined by 
		$$H_I^i(R)=H^i( \Check{C}^{\bullet}(\underline f,R)).$$
		Since $\mathcal{B}_m$ is stable under extension, using point $2$ and $3$, we see that $H^i_I(R)\in \mathcal{B}_m$.
	\end{enumerate}
    \end{point}
    Let us recall the following multiplicity bound on the multiplicity of localization of a holonomic $A_m(K)$-modules.
    \begin{theorem}\label{Multiplicity bound}
        Let $M$ be a holonomic $A_m(K)$-module. If $f\in R=k[x_1,\ldots,x_m]$, then
        $$e(M_f)\leq e(M)(1+\deg f)^m.$$
    \end{theorem}
    \begin{proof}
        The proof follows from \cite[Chapter 1, Theorem 5.19]{Bjo-79}.
    \end{proof}
    Now we define Lyubeznik functors.
  \begin{point}\label{lyubeznik functor}
   \normalfont
 \textbf{Lyubeznik functors:}
As before assume that $R$ is a commutative Noetherian ring and let $X = \Spec(R)$. Let $Y$ be a locally closed subset of $X$. If $M$ is an $R$-module and  $Y$ is a locally closed subscheme of $\Spec(R)$, we denote by $H^i_Y(M)$ the $i$-th local cohomology module of $M$ with support in $Y$.  Suppose  $Y = Y_1 \setminus Y_2$ where $Y_2 \subseteq Y_1$ are two closed subsets of $X$, then we have an exact sequence of functors
 $$
 \cdots \rightarrow H^i_{Y_1}(-) \rightarrow H^i_{Y_2}(-) \rightarrow H^i_Y(-) \rightarrow H^{i+1}_{Y_1}(-) \rightarrow\cdots.
 $$
 A Lyubeznik functor $\mathcal{T}$ is any functor of the form $\mathcal{T} = \mathcal{T}_1\circ\mathcal{T}_2 \circ \cdots \circ \mathcal{T}_m$ where every functor $\mathcal{T}_j$  is either $H^i_Y(-)$ for some locally closed subset of $X$ or the kernel, image or cokernel of some arrow in the previous long exact sequence for closed subsets $Y_1,Y_2$ of $X$  such that $Y_2 \subseteq Y_1$.
\end{point}
We should remark that a significant example of a Lyubeznik functor  $\mathcal{T}$ is the repeated local cohomology functor $ H^{i_1}_{I_1}(H^{i_2}_{I_2}(\dots H^{i_r}_{I_r}(R)\dots)$.
\begin{point}\label{countable generation point}
\normalfont
    \textbf{Countable generation:} We say that an $R$-module $M$ is countably generated if there exists a countable set of elements that generates $M$ as an $R$-module. The following results are crucial for proving our results.
    \begin{enumerate}
        \item \cite[Lemma 5.2]{PS-24} Let $0 \rightarrow M_1 \rightarrow M_2 \rightarrow M_3 \rightarrow 0$ be a short exact sequence of $R$-modules. Then, $M_2$ is countably generated if and only if $M_1$ and $M_3$ are countably generated.
        \item \cite[Lemma 5.2]{PS-24} Let $\mathcal{T}$ be a Lyubeznik functor on $\Mod(R)$. Then, $\mathcal{T}(M)$ is a countably generated $R$-module for any countably generated $R$-module $M$. In particular, $H^i_I(R)$ is a countably generated $R$-module for all ideals $I$ of $R$ and for all $i\geq 0$.
        \item \cite[Lemma 2.3]{Put-16-Pac}  Let $M$ be a countably generated $R$-module. Then, $\Ass_R M$ is a countable set. 
    \end{enumerate}
\end{point}
	\section{Proof of theorem A}
	In this section, we prove \hyperref[main theorem for m]{Theorem A}. First, we prove a small lemma that we believe is already known in the literature. However, we do not have any reference, so we include the proof here.
	\begin{lemma}\label{Existence of Gaolis ext}
		Let $K$ be a field of characteristic 0 and let $\mathfrak{m}$ be a maximal ideal of $R=K[x_1,\ldots,x_m]$. Then there exists a finite Galois extension $\widetilde{K}$ (depending on $\m$) of $K$  such that if $\n$ is a maximal ideal of $\widetilde{R}$ lying over $\m$, then $\n $ is of the form $(x_1-c_1,\ldots,x_m-c_m)$ for some $c_i\in \widetilde{K}$.
		
	\end{lemma}
\begin{proof}
	Let $\m$ be a maximal ideal of $K[x_1,\ldots,x_m]$. Then, by \cite[Theorem 5.1]{Mat-86}, we can write $\m=(f_1(x_1),f_2(x_1,x_2),\ldots,f_m(x_1,\ldots,x_m))$ where $f_i(x_1,\ldots,x_i)$ is monic in $x_i$.

    Let $S_1$ be the set of zeroes of the polynomial $f_1(x_1)$. For $a\in S_1$, let $S_2^a$ be the zero set of $f_2(a,x_2)$. Inductively for $3\leq i\leq m$ and for $a_1\in S_1$, $a_j\in S_j^{a_1,\ldots,a_{j-1}}$, we define $S_i^{a_1,\ldots,a_{i-1}}$ to be the set of zeroes of the polynomial $f_i(a_1,\ldots,a_{i-1},x_m)$. A careful look at the proof of \cite[Theorem 5.1]{Mat-86} shows that the polynomials $f_i(x_1,\ldots,x_i)$ are chosen such a way that these $f_i(a_1,\ldots,a_{i-1},x_i)$ are irreducible.
	
	  We claim that our required field is $\widetilde{K}$
      which is obtained by adjoining the elements of the set $S_1\cup
(\bigcup_{i=2}^{m} S_i^{a_1,\dots,a_{i-1}})
$ with $K$.
	
	Now $\m\widetilde{K}[x_1,\ldots,x_m]=(f_1,\ldots,f_m)\widetilde{K}[x_1,\ldots,x_m]$. Let $\n$ be a maximal in $\widetilde{K}[x_1,\ldots,x_m]$ lying over $\m$. Let $\n= (g_1(x_1),g_2(x_1,x_2),\ldots,g_m(x_1,\ldots,x_m))$.
	
	We note that $\n\supseteq \m\widetilde{K}[x_1,\ldots,x_m]$. Similarly, using the roots  of $g_i$ we can form a field, say $\widetilde{\widetilde{K}}$ which is a finite extension of $\widetilde{K}$. Let $\n'$ be a maximal ideal of $\widetilde{\widetilde{K}}[x_1,\ldots,x_m]$ lying over $\n$. Now $\n'\supseteq \n\widetilde{\widetilde{K}}[x_1,\ldots,x_m]\supseteq \m\widetilde{\widetilde{K}}[x_1,\ldots,x_m]$. Therefore, $Z(\n')\subseteq Z(\n\widetilde{\widetilde{K}}[x_1,\ldots,x_m])\subseteq Z(\m\widetilde{\widetilde{K}}[x_1,\ldots,x_m])$. Let $(\alpha_1,\ldots,\alpha_m)\in Z(\n\widetilde{\widetilde{K}}[x_1,\ldots,x_m])\subseteq Z(\m\widetilde{\widetilde{K}}[x_1,\ldots,x_m])$. By construction, $\alpha_1\in S_1$ and $\alpha_j\in S_j^{\alpha_1,\ldots,\alpha_{j-1}}$ for $j\geq 2$. We note that $\alpha_i\in \widetilde{K}$ for all $1\leq i\leq m$. Since for all $1\leq i\leq m$, $g_i(\alpha_1,\ldots,\alpha_{i-1},x_i)$ is irreducible, we can write $g_i(\alpha_1,\ldots,\alpha_{i-1},x_i)=r_i(x_i-\alpha_i)$. Also when $i\geq 2$, division algorithm implies $g_i(x_1,\ldots,x_i)=(x_1-\alpha_1)\cdots(x_{i-1}-\alpha_{i-1})Q_i(x_1\ldots,x_i)+g_i(\alpha_1,\ldots,\alpha_{i-1},x_i)=(x_1-\alpha_1)\cdots(x_{i-1}-\alpha_{i-1})Q_i(x_1\ldots,x_i)+r_i(x_i-\alpha_i)$. Therefore, $\n=(x_1-\alpha_1,\ldots,x_m-\alpha_m)$. This proves our result.
\end{proof}
	The following remark is important.
	\begin{remark}\normalfont\label{Mult rem}
		Let $K_1\subseteq K_2$ be two fields. Let $M$ be $A_m(K_1)$-module. Then $M\otimes_{K_1}K_2$ is $A_m(K_2)$-module and $\dim(M)=\dim(M\otimes_{K_1}K_2)$. Also if $U_1\subsetneq U_2$ are $A_m(K_1)$-submodules of $M$, then $U_1\otimes_{K_1}K_2\subsetneq U_2\otimes_{K_1}K_2$ are $A_m(K_2)$-submodules of $M\otimes_{K_1}K_2$. So $l_{A_m(K_1)}(M)\leq l_{A_m(K_2)}(M\otimes_{K_1}K_2)$. Also, if $\Omega$ is a good filtration on $M$, then $\Omega\otimes K_2$ is a good filtration on $M\otimes K_2$ and hence  $e_{A_m(K_1)}(M)= e_{A_m(K_2)}(M\otimes_{K_1}K_2)$.
		\end{remark}
        We now prove  \hyperref[main theorem for m]{Theorem A}, which we restate for the reader’s convenience.
	\begin{theorem}\label{Bounded bass no for m}
		Let $R=K[x_1,\ldots,x_m]$ where $K$ is a field of characteristic $0$. Assume that $T(R)= H^{i_1}_{I_1}(H^{i_2}_{I_2}(\dots H^{i_r}_{I_r}(R)\dots)$. Suppose $\injdim T(R)=c$. Then for all $i=0,1,\ldots,c$; $$\mu_i(\m ,T(R))\leq \binom{m}{i}(1+i)^n e_{A_m(K)}T(R) $$ for all  maximal ideal $\m$ of R.
	\end{theorem}
	\begin{proof}
		Let $\widetilde{K}$	be a finite Galois extension of $K$ as in \ref{Existence of Gaolis ext}. Assume that the Galois group is $G$. Then $G$ acts  on $\widetilde{R}=\widetilde{K}[x_1,\ldots,x_m]$ and $\widetilde{R}^G=R$. Since $T(R)=T(\widetilde{R}^G)$ so by Theorem \ref{Bass number relation}, $\mu_i(\m ,T(R))=\mu_i(\m ,T(\widetilde{R}^G))=\mu_i(\n ,T(\widetilde{R}))$ where $\n$ is any maximal in $\widetilde{R}$ lying over $\m$. By \ref{Existence of Gaolis ext}, $\n $ is of the form $(x_1-c_1,\ldots,x_m-c_m)$ for some $c_i\in \widetilde{K}$. From the proof  of \cite[Proposition 1.3]{Put-14}, 
        \begin{align}\label{equation in maximal ideal theorem}
		    \mu_i(\n ,T(\widetilde{R}))\leq \binom{m}{i}(1+i)^n e_{A_m(\widetilde{K})}(T(\widetilde{R})).
            \end{align}
            Note that $T(\widetilde{R})=T(R)\otimes_K \widetilde{K}$. Using Remark \ref{Mult rem}, $e_{A_m(\widetilde{K})}(T(\widetilde{R}))= e_{A_m(K)}T(R)$.  Hence, from equation \ref{equation in maximal ideal theorem}, $\mu_i(\n ,T(\widetilde{R}))\leq  \binom{m}{i}(1+i)^n e_{A_m(K)}T(R)$. This proves our result.
	\end{proof}
    \section{Proof of Theorem B}
     Before proving  our result, we recall a result due to Goto and Watanabe from \cite{GW-78}.
     \begin{theorem}\cite[Theorem 1.1.2]{GW-78} \label{goto and watanabe result}
         Let $M$ be a graded $R$-module and let $P$ be a non-graded prime ideal of $R$. Then $\mu_0(P,M)=0$ and $\mu_{i+1}(P,M)=\mu_i(P^*,M)$ for every integer $i\geq 0$. Here, $P^*$ is the largest homogeneous prime ideal of $R$ contained in $P$.
     \end{theorem}
     Now we prove \hyperref[main theorem for P but graded]{Theorem B}.
	\begin{theorem}\label{proof of graded case m-1}
	Let $K$ be a field of characteristic 0 and $R=K[x_1,\ldots,x_m]$ be standard graded i.e., $\deg(x_i)=1$ for all $i$.  Let $T(R)= H^{i_1}_{I_1}(H^{i_2}_{I_2}(\dots H^{i_r}_{I_r}(R)\dots)$ for some homogeneous ideals $I_1,I_2,\ldots,I_r$ of $R$. Then, there exists some $B> 0$ $\left(\text{depending on}\  T(R)\right)$ such that $$\mu_i(P,T(R))\leq B $$ for all  graded prime ideal $P$ of $ R $ of height $m-1$.
	\end{theorem}
	\begin{proof}
Let $\bar{\m}$ be a maximal ideal of $R/P$ which is not graded. By Theorem \ref{goto and watanabe result}, we get $\mu_i(P,T(R))=\mu_{i+1}(\m,T(R))$. Hence, the result follows from Theorem $\ref{Bounded bass no for m}$.
	\end{proof}
     \section{Proof of Lemma C and Theorem D}
   In this section, we prove \hyperref[main theorem non graded]{Lemma C}, and as an immediate consequence, we obtain \hyperref[main theorem corollary]{Theorem D}. Let us first highlight the following important remark before proceeding to the proof.
   \begin{remark}\label{remark for 1st bass number}
   \normalfont
Let $R=K[x_1,\ldots,x_m]$ be a polynomial ring over a field of characteristic $0$ and let $P$ be a prime ideal of $R$ of height $g$ where $1\leq g\leq m-1$. Let $Q$ be a prime  ideal of $R$ containing $P$ such that $\height Q=g+1$. Note $R_Q/PR_Q$ is Cohen-Macaulay ring of dimension $1$. Therefore,
        from \cite[Theorem 1.2, Proposition 2.2]{Put-14}, we get $\mu_1(Q,H^g_P(R))=\mu_1(QR_Q,H^g_{PR_Q}(R_Q))=1$. 
    \end{remark}
  We now present the proof.
	\begin{lemma}\label{Result of mu_0}
	Let $K$ be an uncountable algebraically closed field of characteristic 0 and $R=K[x_1,\ldots,x_m]$. Let $\mathcal{T}$ be a Lyubeznik functor on $\Mod(R)$ . Then, $$\mu_0(P,\mathcal{T}(R))\leq 2^mm\ e(\mathcal{T}(R))$$ for all prime ideal $P$ of $R$ of height $m-1$.
	\end{lemma}
	\begin{proof}
	    Let $M=\mathcal{T}(R)$ and $P$ be a prime ideal of $R$ height $m-1$. By \cite[Lemma 1.4, Theorem 3.4(b)]{Lyu-93}, $\mu_0(P,M)=\mu_0(P,H^0_P(M))$. We note that $e(H^0_P(M))\leq e(M)$ and $\Ass_R(H^0_P(M))$ is a finite set.  If $P\notin \Ass H^0_P(M)$, then there is nothing to prove since $\mu_0(P,M)=0$. Hence, assume that $$\Ass_R(H^0_P(M))=\{P,\m_1,\ldots,\m_s\}$$ where $\m_i$ are maximal ideals. Consider the exact sequence
        $$0\rightarrow \Gamma_{\m_1\cdots \m_s}(H^0_P(M))\rightarrow H^0_P(M)\rightarrow N\rightarrow 0.$$
Clearly $e(N)\leq e(H^0_P(M))$ and $N=\mathcal{F}(R)$ for some Lyubeznik functor $\mathcal{F}$. By \cite[Theorem 3.4(b)]{Lyu-93}, there exists some finite $l>0$ such that
$N_P=H^{m-1}_{PR_P}(R_P)^l=E(\kappa(P))^l$ as $R_P$-module and hence as $R$-module. Note that $l=\mu_0(P,N)=\mu_0(P,M)$. Let $C=N\cap H^{m-1}_{P}(R)^l$. Therefore, $C_P=N_P\cap H^{m-1}_{PR_P}(R_P)^l= N_P=H^{m-1}_{PR_P}(R_P)^l$. Consider the exact sequence \begin{align}\label{equation w1}
     0\rightarrow C\rightarrow H^{m-1}_{P}(R)^l\rightarrow W_1\rightarrow 0.
     \end{align}
Note that $W_1$ is $P$-torsion and $(W_1)_P=0$ so $\Supp_R W_1\subseteq \maxsp{R}\ \cap \variety(P)$. Since $ H^{m-1}_{P}(R)$ is countably generated, $W_1$ is also countably generated (See, \ref{countable generation point}). Hence, $\Ass_R W_1=\Supp_R W_1$ is a countable subset of $\maxsp{R}\ \cap \variety(P)$. 

Similarly, the exact sequence \begin{align}\label{equation w2}
    0\rightarrow C\rightarrow N\rightarrow W_2\rightarrow 0\end{align} implies that $\Ass_R W_2=\Supp_R W_2$ is a countable subset of $\maxsp{R}\ \cap \variety(P)$. 

We note that $\maxsp{R}\ \cap \variety(P)$ is an uncountable set. This is because by Noether normalization $K[z]\hookrightarrow R/P$ is an integral extension and since $K[z]$ has uncountably many maximal ideals, $R/P$ also has uncountably many maximal ideals. Therefore, there is some $\n\in \maxsp{R}\ \cap \variety(P)$ such that ${(W_1)}_{\n}= {(W_2)}_{\n}=0$. Hence, from equation \ref{equation w1}, we get $C_\n=H^{m-1}_{P}(R)^l_{\n}$. Remark \ref{remark for 1st bass number} implies that $\mu_1(\n,C)=\mu_1(\n R_{\n},C_{\n})=l$. Also since ${(W_2)}_{\n}=0$, from equation \ref{equation w2} we get $C_\n=N_\n$. By \cite[Proposition 1.3]{Put-14}, $l=\mu_1(\n,C)=\mu_1(\n R_{\n},C_{\n})=\mu_1(\n R_{\n},N_{\n})=\mu_1(\n,N)\leq 2^mm\ e(N)\leq 2^mm\ e(H^0_P(M))\leq 2^mm\ e(M)$. This completes our proof.
\end{proof}
As a consequence of the preceding lemma, we next examine the higher Bass numbers of $\mathcal{T}(R)$.  We reiterate that, as discussed in the introduction, there exists a value of $t$ such that $H_{m-1}(t)$ is an uncountable set. The next result states that there is a bound of Bass numbers of $\mathcal{T}(R)$ for the primes on the uncountable set $H_{m-1}(t)$.
\begin{theorem} \phantomsection\label{corollary for m-1}
   Let $R=K[x_1,\ldots,x_m]$ where $K$ is an uncountable algebraically closed field of characteristic $0$.  Let $\mathcal{T}$ be a Lyubeznik functor on $\Mod(R)$. Then, for all $P\in H_{m-1}(t)$, $$\mu_i(P,\mathcal{T}(R))\leq \phi_i^{m-1}(e(\mathcal{T}(R)),t)$$ 
   where $$\phi_i^{m-1}(u,v)=2^mm\binom{m}{i}u(1+iv)^m.$$
\end{theorem}
\begin{proof}
    Let $M=\mathcal{T}(R)$. By \cite[Lemma 1.4, Theorem 3.4(b)]{Lyu-93}, $\mu_i(P,M)=\mu_0(P,H^i(M))$. By Lemma \ref{Result of mu_0}, $\mu_0(P,H^i(M))\leq 2^mm\ e(H^i_P(M))$. Note that $H^i_P(M)=H^i_J(M)$ where $J=(f_1,\ldots,f_m)$ with $\deg f_j\leq t$ for all $j$. $H^i_J(M)$ is the cohomology of the following Cech complex
    $$
		\Check{C}^{\bullet}(\underline f,M)=\Check{C}^{\bullet}:\  0\rightarrow M\rightarrow \bigoplus_iM_{f_i}\rightarrow \bigoplus_{i,j}M_{f_if_j}\rightarrow\ldots \rightarrow M_{f_1\ldots f_m}\rightarrow 0.
		$$
        Hence, $e(H^i_P(M))\leq e(\Check{C}^i)\leq \binom{m}{i}e(M)(1+it)^m$. The last inequality follows from Theorem \ref{Multiplicity bound}. Therefore, $\mu_i(P,M)\leq 2^mm\binom{m}{i}e(M)(1+it)^m$. This completes the proof of our theorem.
	\end{proof}
	\section{Proof of Proposition E and Theorem F}
    In this section, our goal is to  prove  \hyperref[bounded by some function]{Theorem F}. Before that, let us prove \hyperref[Proposition about S_g]{Proposition E}. For the reader's convenience, we again recall the definition of the sets $\mathcal{S}_g(t)$ as stated below. 
    \begin{definition}
Define $\mathcal{S}_{m-1}(t):=H_{m-1}(t)$. If $S_{g+1}(t)$ is defined for some $1\leq g\leq m-2$, then we define $\mathcal{S}_g(t)$ in the following way:
 $$\mathcal{S}_g(t):=\{P\in H_g(t)\mid\  C^P_g(t)\  \text{is uncountable}\}.$$ where 
 $$C^P_g(t):=\{Q\in S_{g+1}(t) \mid Q\supsetneq P\}$$
 \end{definition} 
 We note that $H_{m-1}(t)\subseteq H_{m-1}(t+1)$ for all $t\geq 1$ and $\bigcup_{t\geq 1}H_{m-1}(t)=\Spec_{m-1}(R)=\{P\in \Spec(R)\ |\  \height P=m-1\}$. We now show via the following  proposition that for $1\leq g\leq m-2$, the sets $\mathcal{S}_g(t)$ have same properties.
    \begin{proposition}\label{proposition of s_g proof}
        Let $1\leq g\leq m-2$. Then,
  \begin{enumerate}[\rm (1)]
      \item Let $P\in \mathcal{S}_g(t) $. Then, $C^P_g(t)\subseteq C^P_g(t+1)$ for all $t\geq 1$ and $\bigcup_{t\geq 1}C^P_g(t)=\Spec_{g+1}(R)\cap \variety(P)$;
      \item $\mathcal{S}_g(t)\subseteq \mathcal{S}_g(t+1)$ for all $t\geq 1$ and $\bigcup_{t\geq 1} \mathcal{S}_g(t)=\Spec_g(R)$.
      \end{enumerate}
    \end{proposition}
    \begin{proof}
        Let $g=m-2$. Then $C^P_{m-2}(t)=\{Q\in H_{m-1}(t)\mid Q\supsetneq P\}$. Clearly, $C^P_{m-2}(t)\subseteq C^P_{m-2}(t+1)$ for all $t\geq 1$ and $\bigcup_{t\geq 1}C^P_{m-2}(t)=\Spec_{m-1}(R)\cap \variety(P)$. This proves (1) and we now prove (2). Since $\Spec_{m-1}(R)\cap \variety(P)$ is uncountable, there exists some $t_0\geq 1$ such that $C^P_{m-2}(t_0)$ is uncountable and therefore $C^P_{m-2}(t)$  is uncountable for all $t\geq t_0$. We claim that $\bigcup_{t\geq 1}\mathcal{S}_{m-2}(t)=\bigcup_{t\geq 1}H_{m-2}(t)=\Spec_{m-2}(R)$. This is because if $P\in \bigcup_{t\geq 1}H_{m-2}(t)$, then $P\in H_{m-2}(r_0)$ for some $r_0$. Choose a sufficiently large $s_0$ (bigger than $r_0$)  so that  $C^P_{m-2}(s_0)$  is uncountable. Then, $ P\in H_{m-2}(r_0)\subseteq H_{m-2}(s_0)$ and this implies $ P\in \mathcal{S}_{m-2}(s_0)$. Hence, $\bigcup_{t\geq 1}\mathcal{S}_{m-2}(t)=\bigcup_{t\geq 1}H_{m-2}(t)$.

        Now assume that the result holds for some $2\leq g+1\leq m-2$ and we prove it for $g$. We have $$C^P_g(t)=\{Q\in S_{g+1}(t) \mid Q\supsetneq P\}$$ and  $$\mathcal{S}_g(t)=\{P\in H_g(t)\mid\  C^P_g(t)\  \text{is uncountable}\}.$$ Since for all $t\geq 1$, $\mathcal{S}_{g+1}(t)\subseteq \mathcal{S}_{g+1}(t+1)$, we get that $C^P_g(t)\subseteq C^P_g(t+1)$. Also, $\bigcup_{t\geq 1}C^P_g(t)=\bigcup_{t\geq 1}\mathcal{S}_{g+1}(t)\cap \variety(P)=\Spec_{g+1}(R)\cap \variety(P)$. This proves (1) of the proposition, and we proceed to prove (2). Since $\Spec_{g+1}(R)\cap \variety(P)$ is uncountable, there exists some $t_1\geq 1$ such that $C^P_{g}(t_1)$ is uncountable and therefore $C^P_{g}(t)$  is uncountable for all $t\geq t_1$. Now a similar argument as in $m-2$ case proves that $\bigcup_{t\geq 1}\mathcal{S}_{g}(t)=\bigcup_{t\geq 1}H_{g}(t)$. In fact, if $P\in \bigcup_{t\geq 1}H_{g}(t)$, then $P\in H_{g}(r_1)$ for some $r_1$. Choose a sufficiently large $s_1$ (bigger than $r_1$)  so that  $C^P_{g}(s_1)$  is uncountable. Then, $ P\in H_{g}(r_1)\subseteq H_{g}(s_1)$ and this implies $ P\in \mathcal{S}_{g}(s_1)$. Hence, $\bigcup_{t\geq 1}\mathcal{S}_{g}(t)=\bigcup_{t\geq 1}H_{g}(t)=\Spec_g(R)$. This completes the proof of the proposition.
    \end{proof}
    Finally, we prove the main result of this paper.
	\begin{theorem}\label{proof of main theorem}
	   Let $1\leq g\leq m-1$ and let $R=K[x_1,\ldots,x_m]$ where $K$ is an uncountable algebraically closed field of characteristic $0$. Let $M=\mathcal{T}(R)$ for some Lyubeznik functor $\mathcal{T}$. Then, there exists some function $\phi^g_i: \mathbb{N}^2\rightarrow \mathbb{N}$ which is monotonic in both the variables such that  $\mu_i(P,M)\leq \phi^g_i(e(M),t)$ for all $P\in \mathcal{S}_{g}(t)$.
	\end{theorem}
	\begin{proof}
    If $g=m-1$, the result follows from Theorem \ref{corollary for m-1}. Next, we proceed in two steps to prove the theorem. \medskip
    
    \noindent
\textbf{Step 1:} Let $g=m-2$.\medskip
     
	    Let $i=0$. By \cite[Lemma 1.4, Theorem 3.4(b)]{Lyu-93}, $\mu_0(P,M)=\mu_0(P,H^0_P(M))$. We note that $e(H^0_P(M))\leq e(M)$ and $\Ass_R(H^0_P(M))$ is a finite set.  If $P\notin \Ass H^0_P(M)$, then there is nothing to prove since $\mu_0(P,M)=0$. Hence, assume that $$\Ass_R(H^0_P(M))=\{P,Q_1\ldots,Q_r\}$$ where $\height Q_i\geq m-1$. Consider the exact sequence
        $$0\rightarrow \Gamma_{Q_1\cdots Q_r}(H^0_P(M))\rightarrow H^0_P(M)\rightarrow N\rightarrow 0.$$
Clearly, $e(N)\leq e(H^0_P(M))$ and $N=\mathcal{F}(R)$ for some Lyubeznik functor $\mathcal{F}$. By \cite[Theorem 3.4(b)]{Lyu-93}, there exists some finite $l>0$ such that $N_P=H^{m-2}_{PR_P}(R_P)^l=E(\kappa(P))^l$ as $R_P$-module and hence as $R$-module. Note that $l=\mu_0(P,N)=\mu_0(P,M)$. Let $C=N\cap H^{m-2}_{P}(R)^l$. Therefore, $C_P=N_P\cap H^{m-2}_{PR_P}(R_P)^l= N_P=H^{m-2}_{PR_P}(R_P)^l$. Consider the exact sequence \begin{align}\label{equation w1'}
     0\rightarrow C\rightarrow H^{m-2}_{P}(R)^l\rightarrow W_1\rightarrow 0.
     \end{align}
Since $ H^{m-2}_{P}(R)$ is countably generated, $W_1$ is also countably generated (See, \ref{countable generation point}). We note that $\Supp_R W_1\subseteq \{Q: \height Q\geq m-1, P\subseteq Q\}$. Let $L(W_1):=\{Q: \height Q= m-1, Q\supseteq P,\  Q\in \Supp_R W_1\}\subseteq \Ass_R W_1$ is a countable set.

Similarly, the exact sequence 
\begin{align}\label{equation w2'}
     0\rightarrow C\rightarrow N\rightarrow W_2\rightarrow 0
     \end{align}
	implies $L(W_2):=\{Q: \height Q= m-1, Q\supseteq P,\  Q\in \Supp_R W_2\}\subseteq \Ass_R W_2$ is a countable set. Since $C^P_{m-2}(t)$ is uncountable, there is some $Q\in C^P_{m-2}(t)$ such that $(W_1)_Q=(W_2)_Q=0$. Hence, equations \ref{equation w1'} and \ref{equation w2'} imply $C_Q=H^{m-2}_{P}(R)^l_Q=N_Q$. By Remark \ref{remark for 1st bass number}, $l=\mu_1(Q,C)=\mu_1(Q,N)\leq 2^mm^2e(N)(1+t)^m \leq 2^mm^2e(M)(1+t)^m $. The second last inequality follows from Theorem \ref{corollary for m-1}. Let $\phi^{m-2}_0(u,v):=2^mm^2u(1+v)^m$.

    Now suppose $i\geq 1$. Then, $\mu_i(P,M)=\mu_0(P,H^i_P(M))$. Therefore, from the $i=0$ case we get  $\mu_i(P,M)\leq \phi^{m-2}_0(e(H^i_P(M),t) $. Again, from the proof of \ref{corollary for m-1}, we have $e(H^i_P(M))\leq h(e(M),t)$ where $h(u,v)=\binom{m}{i}u(1+iv)^m$. Hence,  $\mu_i(P,M)\leq \phi^{m-2}_i(e(M),t)$ where $\phi^{m-2}_i(u,v)=\phi^{m-2}_0(h(u,v),v)$. This completes the proof of $m-2$ case.\medskip

    \noindent
\textbf{Step 2:}  Now assume that the result holds for some $2\leq g+1\leq m-2$ and we prove it for $g$.\medskip

Let $i=0$. An argument similar to that in the case $m-2$ gives us that there is some $Q\in C^P_g(t)$ such that $\mu_0(P,M)=\mu_1(Q,N)\leq \phi^{g+1}_1(e(N),t)\leq \phi^{g+1}_1(e(M),t) $ where $\phi^{g+1}_1$ is some monotonic function in both variables. Choose $\phi^{g}_0$ to be $\phi^{g+1}_1$.

Now suppose $i\geq 1$. Then, $\mu_i(P,M)=\mu_0(P,H^i_P(M))$. Therefore, from the $i=0$ case we get  $\mu_i(P,M)\leq \phi^g_0(e(H^i_P(M),t) $. Again, from the proof of \ref{corollary for m-1}, we have $e(H^i_P(M))\leq h(e(M),t)$ where $h(u,v)=\binom{m}{i}u(1+iv)^m$. Hence,  $\mu_i(P,M)\leq \phi^g_i(e(M),t)$ where $\phi^g_i(u,v)=\phi^g_0(h(u,v),v)$.

This completes the proof of the theorem.
	\end{proof}

	\section*{Acknowledgement}
	The first named author thanks the Government of India for support through the Prime Minister's Research Fellowship (PMRF ID: 1303161).

\providecommand{\bysame}{\leavevmode\hbox to3em{\hrulefill}\thinspace}
\providecommand{\MR}{\relax\ifhmode\unskip\space\fi MR }
\providecommand{\MRhref}[2]{
  \href{http://www.ams.org/mathscinet-getitem?mr=#1}{#2}
}

\end{document}